\documentclass[twoside,leqno,10pt, A4]{amsart}
\usepackage{amsfonts}
\usepackage{amsmath}
\usepackage{amscd}
\usepackage{amssymb}
\usepackage{amsthm}
\usepackage{amsrefs}
\usepackage{latexsym}
\usepackage{mathrsfs}
\usepackage{bbm}
\usepackage{enumerate}
\usepackage{graphicx}

\usepackage{amsfonts}
\usepackage{amsmath}
\usepackage{amscd}
\usepackage{amssymb}
\usepackage{amsthm}
\usepackage{amsrefs}
\usepackage{latexsym}
\usepackage{mathrsfs}
\usepackage{bbm}
\usepackage{amscd}
\usepackage{amssymb}
\usepackage{amsthm}
\usepackage{amsrefs}
\usepackage{latexsym}
\usepackage{mathrsfs}
\usepackage{bbm}
\usepackage{enumerate}
\usepackage{graphicx}
\allowdisplaybreaks
\usepackage{color}
\begin{document}

\newtheorem{theorem}[subsection]{Theorem}
\newtheorem{proposition}[subsection]{Proposition}
\newtheorem{lemma}[subsection]{Lemma}
\newtheorem{corollary}[subsection]{Corollary}
\newtheorem{conjecture}[subsection]{Conjecture}
\newtheorem{prop}[subsection]{Proposition}
\numberwithin{equation}{section}
\newcommand{\mr}{\ensuremath{\mathbb R}}
\newcommand{\mc}{\ensuremath{\mathbb C}}
\newcommand{\dif}{\mathrm{d}}
\newcommand{\intz}{\mathbb{Z}}
\newcommand{\ratq}{\mathbb{Q}}
\newcommand{\natn}{\mathbb{N}}
\newcommand{\comc}{\mathbb{C}}
\newcommand{\rear}{\mathbb{R}}
\newcommand{\prip}{\mathbb{P}}
\newcommand{\uph}{\mathbb{H}}
\newcommand{\fief}{\mathbb{F}}
\newcommand{\majorarc}{\mathfrak{M}}
\newcommand{\minorarc}{\mathfrak{m}}
\newcommand{\sings}{\mathfrak{S}}
\newcommand{\fA}{\ensuremath{\mathfrak A}}
\newcommand{\mn}{\ensuremath{\mathbb N}}
\newcommand{\mq}{\ensuremath{\mathbb Q}}
\newcommand{\half}{\tfrac{1}{2}}
\newcommand{\f}{f\times \chi}
\newcommand{\summ}{\mathop{{\sum}^{\star}}}
\newcommand{\chiq}{\chi \bmod q}
\newcommand{\chidb}{\chi \bmod db}
\newcommand{\chid}{\chi \bmod d}
\newcommand{\sym}{\text{sym}^2}
\newcommand{\hhalf}{\tfrac{1}{2}}
\newcommand{\sumstar}{\sideset{}{^*}\sum}
\newcommand{\sumprime}{\sideset{}{'}\sum}
\newcommand{\sumprimeprime}{\sideset{}{''}\sum}
\newcommand{\sumflat}{\sideset{}{^\flat}\sum}
\newcommand{\shortmod}{\ensuremath{\negthickspace \negthickspace \negthickspace \pmod}}
\newcommand{\V}{V\left(\frac{nm}{q^2}\right)}
\newcommand{\sumi}{\mathop{{\sum}^{\dagger}}}
\newcommand{\mz}{\ensuremath{\mathbb Z}}
\newcommand{\leg}[2]{\left(\frac{#1}{#2}\right)}
\newcommand{\muK}{\mu_{\omega}}
\newcommand{\thalf}{\tfrac12}
\newcommand{\lp}{\left(}
\newcommand{\rp}{\right)}
\newcommand{\Lam}{\Lambda_{[i]}}
\newcommand{\lam}{\lambda}
\def\L{\fracwithdelims}
\def\om{\omega}
\def\pbar{\overline{\psi}}
\def\phis{\phi^*}
\def\lam{\lambda}
\def\lbar{\overline{\lambda}}
\newcommand\Sum{\Cal S}
\def\Lam{\Lambda}
\newcommand{\sumtt}{\underset{(d,2)=1}{{\sum}^*}}
\newcommand{\sumt}{\underset{(d,2)=1}{\sum \nolimits^{*}} \widetilde w\left( \frac dX \right) }

\newcommand{\hf}{\tfrac{1}{2}}
\newcommand{\af}{\mathfrak{a}}
\newcommand{\Wf}{\mathcal{W}}

\newtheorem{mylemma}{Lemma}
\newcommand{\intR}{\int_{-\infty}^{\infty}}

\theoremstyle{plain}
\newtheorem{conj}{Conjecture}
\newtheorem{remark}[subsection]{Remark}

\makeatletter
\def\widebreve{\mathpalette\wide@breve}
\def\wide@breve#1#2{\sbox\z@{$#1#2$}%
     \mathop{\vbox{\m@th\ialign{##\crcr
\kern0.08em\brevefill#1{0.8\wd\z@}\crcr\noalign{\nointerlineskip}%
                    $\hss#1#2\hss$\crcr}}}\limits}
\def\brevefill#1#2{$\m@th\sbox\tw@{$#1($}%
  \hss\resizebox{#2}{\wd\tw@}{\rotatebox[origin=c]{90}{\upshape(}}\hss$}
\makeatletter

\title[Shifted moments of cubic and quartic Dirichlet $L$-functions]{Shifted moments of cubic and quartic Dirichlet $L$-functions}

\author[P. Gao]{Peng Gao}
\address{School of Mathematical Sciences, Beihang University, Beijing 100191, China}
\email{penggao@buaa.edu.cn}

\author[L. Zhao]{Liangyi Zhao}
\address{School of Mathematics and Statistics, University of New South Wales, Sydney NSW 2052, Australia}
\email{l.zhao@unsw.edu.au}

\begin{abstract}
 We establish upper bounds for shifted moments of cubic and quartic Dirichlet $L$-functions under the generalized Riemann hypothesis.  As applications, we prove  bounds for moments of cubic and quartic Dirichlet character sums.
\end{abstract}

\maketitle



\section{Introduction}
\label{sec 1}

The study of character sums is ubiquitous in analytic number theory, as their applications abound.  A lot of knowledge has been gained on the moments of character sums. For example, A. J. Harper \cite{Harper23} showed that the low moments of zeta sums (and also character sums) have ``better than square-root cancellation". \newline
  
For high moments, M. Munsch \cite{Munsch17} showed that upper bounds for the shifted moments of the family of Dirichlet $L$-functions to a fixed modulus can be used to prove bounds for moments of character sums. In \cite{Szab}, B. Szab\'o  strengthened the result of Munsch and established, under the generalized Riemann hypothesis (GRH), that for a large fixed modulus $q$, any positive integer $k$, real tuples ${\bf a} =  (a_1, \ldots, a_k), {\bf t} =  (t_1, \ldots, t_k)$ with $a_j \geq 0$ and $|t_j| \leq q^A$ for a fixed positive real number $A$,
\begin{align}
\begin{split}
\label{eqn:shiftedMoments}
 \sum_{\chi\in X_q^*}\big| L\big( \tfrac12+it_1,\chi \big) \big|^{a_1} \cdots \big| L\big( \tfrac12+it_{k},\chi \big) \big|^{a_{k}} \ll_{\bf{t}, \bf{a}} &  \varphi(q)(\log q)^{(a_1^2+\cdots +a_{k}^2)/4} \prod_{1\leq j<l\leq k}  \big|\zeta(1+i(t_j-t_l)+\tfrac 1{\log q}) \big|^{a_ja_l/2}.
\end{split}
\end{align}
Here $X_q^*$ denotes the set of primitive Dirichlet characters modulo $q$, $\varphi$ the Euler totient function and $\zeta(s)$ the Riemann zeta function. 
 The authors \cite{G&Zhao24-11} obtained a partial companion lower bounds result, i.e. under GRH, one may replace $\ll_{\bf{t}}$ by  $\gg_{\bf{t}}$ in \eqref{eqn:shiftedMoments} for $k=2$ and primes $q$. \newline
  
Utilizing \eqref{eqn:shiftedMoments}, Szab\'o proved, under GRH, that for a fixed real number $k>2$, a large integer $q$ and $2 \leq Y \leq q$,
\begin{align*}
  \sum_{\chi\in X_q^*}\bigg|\sum_{n\leq Y} \chi(n)\bigg|^{2k} \ll_k \varphi(q) Y^k\min \left( \log Y, \log \frac {2q}{Y} \right)^{(k-1)^2}. 
\end{align*}
The optimality of the above abound is also known from \cite{Szab24}. \newline
  
The approach of using bounds for shifted moments of $L$-functions to estimate moments of the associated character sums has borne more fruits in its applications.  For instance,  the authors \cite{G&Zhao2024-3} used the method in \cites{Sound2009, Harper} to show that, under GRH, for fixed integer $k\geq 1$, positive real numbers $a_1,\ldots, a_{k}, A$ and a real $k$-tuple $t=(t_1,\ldots ,t_{k})$ with $|t_j|\leq  X^A$ for a large real number $X$, 
\begin{align*}
\begin{split}
 \sumprime_{\substack{(d,2)=1 \\ d \leq X}}  \big| L \big(  & \tfrac12+it_1,\chi^{(8d)} \big) \big|^{a_1} \cdots \big| L\big(\tfrac12+it_{k},\chi^{(8d)}  \big) \big|^{a_{k}} \\
 & \ll X(\log X)^{(a_1^2+\cdots +a_{k}^2)/4}  \prod_{1\leq j<l \leq k} \Big|\zeta \Big( 1+i(t_j-t_l)+\tfrac 1{\log X} \Big) \Big|^{a_ia_j/2} \Big|\zeta \Big(1+i(t_j+t_l)+\tfrac 1{\log X} \Big) \Big|^{a_ia_j/2} \\
& \hspace*{5cm} \times\prod_{1\leq j\leq k} \Big|\zeta \Big(1+2it_j+\tfrac 1{\log X} \Big) \Big|^{a^2_i/4+a_i/2},
\end{split}
\end{align*}
 where $\sum'$ denotes a sum over positive square-free integers and $\chi^{(8d)}=\leg {8d}{\cdot}$ is the Kronecker symbol. The result is then used to establish upper bounds for high moments of quadratic character sums, confirming a conjecture of M. Jutila \cite{Jutila2}. \newline

Moreover, another application of bounds on shifted moments is a result of M. J. Curran \cite[Theorem 1.1]{Curran} on shifted moments of the Riemann zeta function.  This result enabled the first-named author \cite{Gao2026} to obtain upper bounds for high moments of zeta sums. \newline
  
The aim of this paper is to establish upper bounds for moments of cubic and quartic Dirichlet character sums by making further use of shifted moment estimates of the corresponding Dirichlet $L$-functions.  Such bounds for the unshifted moments of the families of cubic and quartic Dirichlet $L$-functions have already been obtained in \cite[Theorem 1.1]{G&Zhao2022}.  Also, it is pointed out there that the (shifted) moments of $L$-functions can be used to determine their symmetric type. Here we recall that the density conjecture of N. Katz and P. Sarnak \cite{K&S}  associates a classical compact group to each reasonable family of $L$-functions and the underlying symmetry of each family can be obtained by computing the corresponding $n$-level density of lower-lying zeros. Using this approach, it is shown in \cite{O&S} that the symmetric type of the family of quadratic Dirichlet $L$-functions is unitary symplectic, while it emerges from \cite{HuRu} that the symmetric type of the family of Dirichlet $L$-functions to a fixed modulus is unitary.  Moreover, a result in \cite{G&Zhao2} yields that the families of cubic and quartic Dirichlet $L$-functions are both unitary families. \newline
 
In view of the above discussion, it is expected that the bounds for shifted moments of cubic or quartic Dirichlet $L$-functions should resemble that for the family of Dirichlet $L$-functions to a fixed modulus in \eqref{eqn:shiftedMoments}.  Our first result here lends further credence to this.
\begin{theorem}
\label{t1}
 With the notation as above and the truth of GRH, let $k\geq 1$ be a fixed integer and $a_1,\ldots, a_{k}$, $A$ fixed positive real numbers.  Suppose that $X$ is a large real number and $t=(t_1,\ldots ,t_{k})$ a real $k$-tuple with $|t_j|\leq  X^A$. Then
\begin{align}
\label{cqLbounds}
\begin{split}
  \sum_{\substack{(q,3)=1 \\ q \leq X}} \ \sumstar_{\substack{\chi \shortmod{q} \\ \chi^3 = \chi_0}} \prod^k_{j=1} \big| L\big( \tfrac12 +it_j,\chi\big) \big|^{a_j}  \ll &  X(\log X)^{(a_1^2+\cdots +a_{k}^2)/4}\prod_{1\leq j<l\leq k}  \big|\zeta(1+i(t_j-t_l)+\tfrac 1{\log X}) \big|^{a_ja_l/2},  
\end{split}
\end{align}
  where the asterisk on the sum over $\chi$ means that the sum runs over primitive characters and $\chi_0$ denotes the principal character.  Also, 
\begin{align}
\label{cqLbounds1}
\begin{split}
 \sum_{\substack{(q,2)=1 \\ q \leq X}} \ \sumstar_{\substack{\chi \shortmod{q} \\ \chi^4 = \chi_0}}\prod^k_{j=1} \big| L\big( \tfrac12+it_j,\chi\big) \big|^{a_j}\ll & X(\log X)^{(a_1^2+\cdots +a_{k}^2)/4}\prod_{1\leq j<l\leq k}  \big|\zeta(1+i(t_j-t_l)+\tfrac 1{\log X}) \big|^{a_ja_l/2},
\end{split}
\end{align}
 where the asterisk on the sum over $\chi$ restricts the sum to primitive characters $\chi$ such that $\chi^2$ remains primitive.  The implied constants in both \eqref{cqLbounds} and \eqref{cqLbounds1} depend on $k$, $A$ and the $a_j$'s, but not on $X$ or the $t_j$'s.
\end{theorem}

   We shall prove Theorem \ref{t1} in Section \ref{sec:upper bd} by extending the arguments used in \cite{G&Zhao2022}, which ultimately builds on the method of K. Soundararajan and its refinement by A. J. Harper \cite{Harper} to develop sharp upper bounds for moments of families of $L$-functions under GRH. In order to make Theorem \ref{t1} more wieldy in applications, we need to majorize the Riemann zeta functions appearing in \eqref{cqLbounds}.  To this end, we apply \eqref{mertenstype} in Theorem \ref{t1} to deduce the following.
\begin{corollary}
\label{cor1}
Under the same assumptions, notations and conditions of Theorem~\ref{t1}, we have
\begin{align*}
\begin{split}
  \sum_{\substack{(q,3)=1 \\ q \leq X}} \ \sumstar_{\substack{\chi \shortmod{q} \\ \chi^3 = \chi_0}} \prod^k_{j=1} \big| L\big(\tfrac12+it_j,\chi\big) \big|^{a_j} \ll &  X(\log X)^{(a_1^2+\cdots +a_{k}^2)/4} \prod_{1\leq j<l\leq k} g(|t_j-t_l|)^{a_ja_l/2}, \quad \mbox{and}  \\
  \sum_{\substack{(q,2)=1 \\ q \leq X}} \ \sumstar_{\substack{\chi \shortmod{q} \\ \chi^4 = \chi_0}}\prod^k_{j=1} \big| L\big(\tfrac12+it_j,\chi\big) \big|^{a_j}\ll & X(\log X)^{(a_1^2+\cdots +a_{k}^2)/4} \prod_{1\leq j<l\leq k} g(|t_j-t_l|)^{a_ja_l/2},
\end{split}
\end{align*}
where $g:\mathbb{R}_{\geq 0} \rightarrow \mathbb{R}$ is defined by
\begin{equation*} 
g(x) =\begin{cases}
\log X,  & \text{if } x\leq 1/\log X \text{ or } x \geq e^X, \\
1/x, & \text{if }   1/\log X \leq x\leq 10, \\
\log \log x, & \text{if }  10 \leq x \leq e^{X}.
\end{cases}
\end{equation*}
Here the implied constant depends on $k$, $A$ and the $a_j$'s, but not on $X$ or the $t_j$'s.
\end{corollary}
  
Now Corollary \ref{cor1} can be readily used to obtain bounds concerning cubic or quartic Dirichlet character sums.  For simplicity, we consider smoothed sums. 
\begin{theorem}
\label{quadraticmean}
With the notation as above and the truth of GRH, let $\Psi(t)$ be a non-negative smooth function compactly supported on $(0, \infty)$. For any integer $k \geq 1$ and any real number $m >2$, we have for large $X$, $Y$, 
\begin{align*}
 S_{3,m}(X,Y):=& \sum_{\substack{(q,3)=1 \\ q \leq X}} \ \sumstar_{\substack{\chi \shortmod{q} \\ \chi^3 = \chi_0}} \Big | \sum_{n}\chi(n)\Psi\big(\frac n{Y}\big)\Big |^{2m} \ll XY^m(\log X)^{(m-1)^2}, \quad \mbox{and} \\
 S_{4,m}(X,Y):=& \sum_{\substack{(q,2)=1 \\ q \leq X}} \ \sumstar_{\substack{\chi \shortmod{q} \\ \chi^4 = \chi_0}} \Big | \sum_{n}\chi(n)\Psi\big(\frac n{Y}\big)\Big |^{2m} \ll  XY^m(\log X)^{(m-1)^2}.
\end{align*}
\end{theorem}

Finally, we end this section by noting that it would also be interesting to study analogue of our results for characters of orders higher than four. However, the relation between higher order residue symbols and $n$-th order primitive Dirichlet characters would be more difficult to describe.

\section{Preliminaries}
\label{sec 2}

In this section, we gather several auxiliary results necessary in our proofs.

\subsection{Cubic and quartic Dirichlet characters}
    Let $\omega = \exp (2 \pi i/3)$ and we write $K$ for either the number field $\mq(i)$ or $\mq(\omega)$ with $\zeta_K(s)$ being the corresponding Dedekind zeta function.  Let $N(n)$ denote the norm of $n \in K$ and $r_K$ the residue of $\zeta_K(s)$ at $s = 1$.  Also let $D_K$ be the discriminant of $K$ so that (see \cite[sec 3.8]{iwakow}) $D_{\mq(\omega)}=-3$, $D_{\mq(i)}=-4$.  Throughout the paper, we reserve the letter $p$ for a prime number in $\mz$ and $\varpi$ for a prime in $K$. \newline

Write $\mathcal{O}_K$ for the ring of integers in $K$ and $U_K$ the group of units in $\mathcal{O}_K$.  For $K=\mq(\omega)$, we define the cubic residue symbol $\leg{\cdot}{\varpi}_3$ for any prime $\varpi$, $(\varpi, 3)=1$ by $\leg{a}{\varpi}_3 \equiv a^{(N(\varpi)-1)/3} \pmod{\varpi}$ with $\leg{a}{\varpi}_3 \in \{ 1, \omega, \omega^2 \}$ for any $a \in \mathcal{O}_{K}$, $(a, \varpi)=1$.  If $\varpi | a$, we set $\leg{a}{\varpi}_3 =0$.  For any composite $n$ with $(N(n), 3)=1$, the definition of the cubic symbol is then extended to $\leg{\cdot}{n}_3$ multiplicatively.
Similarly, for $K=\mq(i)$, the quartic residue symbol $\leg{\cdot}{\varpi}_4$ for any prime $\varpi$ prime to $2$ is defined by $\leg{a}{\varpi}_4 \equiv
a^{(N(\varpi)-1)/4} \pmod{\varpi}$ with $\leg{a}{\varpi}_4 \in \{ \pm 1, \pm i \}$ for any $a \in \mathcal{O}_{K}$, $(a, \varpi)=1$.  As before, $\leg{a}{\varpi}_4 =0$ for $\varpi | a$ and we extend the quartic symbol multiplicatively to $\leg{\cdot}{n}_4$ for any composite $n$ with $(N(n), 2)=1$.  Further set $\leg{\cdot }{n}_3=\leg{\cdot }{n}_4=1$ for $n \in U_{K}$. \newline

We quote, from \cite[Lemma 2.4]{G&Zhao2022}, the precise description of primitive cubic and quartic Dirichlet characters.
\begin{lemma}
\label{lemma:cubicclass}
 The primitive cubic Dirichlet characters of conductor $q$ prime to $3$ are of the form $\chi_n:m \rightarrow \leg{m}{n}_3$ for some $n \in \mz[\omega]$, $n \equiv 1 \pmod{3}$, $n$ square-free and not divisible by any rational primes, with norm $N(n) = q$.  The primitive quartic Dirichlet characters of conductor $q$ prime to $2$ such that their squares remain primitive are of the form $\chi_n:m \mapsto \leg{m}{n}_4$ for some $n \in \mz[i]$, $n \equiv 1 \pmod{(1+i)^3}$, $n$ square-free and not divisible by any rational primes, with norm $N(n) = q$.
\end{lemma}

\subsection{Various sums}
\label{sec2.4}

From \cite[Lemma 2.2]{G&Zhao2024-3}, we have the following results on sums over rational primes.
\begin{lemma}
\label{RS} Let $x \geq 2$ and $\alpha \geq 0$. We have, for some constant $b$,
\begin{align}
\label{merten}
\sum_{p\le x} \frac{1}{p} =& \log \log x + b+ O\Big(\frac{1}{\log x}\Big), \\
\label{mertenpartialsummation}
\sum_{p\le x} \frac {\log p}{p} =& \log x + O(1),  \quad \mbox{and} \\
\label{mertenstype}
  \sum_{p\leq x} \frac{\cos(\alpha \log p) }{p}=& \log |\zeta(1+1/\log x+i\alpha)| +O(1)
  \leq
\begin{cases}
\log\log x+O(1)            & \text{if }  \alpha\leq 1/\log x \text{ or } \alpha\geq e^x  ,   \\
\log(1/\alpha)+O(1)        & \text{if }  1/\log x\leq \alpha \leq 10,   \\
\log\log\log \alpha + O(1) & \text{if }   10 \leq \alpha \leq e^x.  
\end{cases}
\end{align}
The estimates in the first two cases of \eqref{mertenstype} are unconditional and the third holds under the Riemann hypothesis.
\end{lemma}

  We define $\delta_{n=\text{cubic}}$ to be $1$ or $0$ depending on whether $n$ is a cube or not, and $\delta_{n=\text{quartic}}$ similarly for fourth powers. 
For any integer $c \in \mz$, set
\begin{align}
\label{gc}
  g(c)=& \prod_{\varpi | c} (1+ N(\varpi)^{-1})^{-1} \prod_{p | c} \Big(1-\frac 1{p^2}\prod_{\varpi | p} (1- N(\varpi)^{-2})^{-1}  \Big )^{-1}.
\end{align}
We follow the convention that an empty product is defined to be $1$.  The same notation $g(c)$ will be used for both $K=\mq(\omega)$ and $\mq(i)$ with the meaning of $\varpi$ varying accordingly.  There should be no ambiguity, as the context makes the meaning clear.  We further define
\begin{equation}
\label{zk}
 c_{K}= r_K \zeta^{-1}_{K}(2) \prod_{(p, D_K)=1}\Big (1-\frac 1{p^2} \prod_{\varpi | p} (1- N(\varpi)^{-2})^{-1} \Big ).
\end{equation}

Let $\Phi$ for a smooth, non-negative function compactly supported on $(1, \infty)$. Recall that the Mellin transform ${\widehat \Phi}(s)$ for any complex number $s$ of $\Phi$ is defined as
\begin{equation*}
{\widehat \Phi}(s) = \int\limits_{0}^{\infty} \Phi(x)x^{s}\frac {\dif x}{x}.
\end{equation*}

The following result on smoothed sums of cubic and quartic characters is quoted from \cite[Lemma 2.5]{G&Zhao2022}.
\begin{lemma} \label{PropDirpoly}
With the notation as above, for large $X$ and any positive integer $c$, we have
\begin{align*}
\begin{split}
\sum_{(q,3)=1} \ \sumstar_{\substack{\chi \shortmod{q} \\ \chi^3 = \chi_0}} \chi(c) \Phi\Big(\frac{q}{X}\Big)=&
\displaystyle \delta_{c=\text{cubic}}c_{\mq(\omega)} {\widehat \Phi}(1) X g(3c) + O( X^{1/2+\varepsilon}c^{1/2+\varepsilon} ), \quad \mbox{and} \\
\sum_{(q,2)=1} \ \sumstar_{\substack{\chi \shortmod{q} \\ \chi^4 = \chi_0}} \chi(c) \Phi\Big(\frac{q}{X}\Big)=&
\displaystyle \delta_{c=\text{quartic}}c_{\mq(i)} {\widehat \Phi}(1) X g(2c) + O( X^{1/2+\varepsilon}c^{1/2+\varepsilon} ).
\end{split}
\end{align*}
Here $c_{\mq(\omega)}$ and $c_{\mq(i)}$ are defined in \eqref{zk} and $g$ in \eqref{gc}.
\end{lemma}

For brevity and notational ease, we write 
\begin{align*} 
\begin{split}
 \sumstar\limits_{\chi, q} =& \sum_{\substack{(q,3)=1 }}\;  \sumstar_{\substack{\chi \shortmod{q} \\ \chi^3 = \chi_0}} \quad \mbox{or} \quad
 \sum_{\substack{(q,2)=1 }}\;  \sumstar_{\substack{\chi \shortmod{q} \\ \chi^4 = \chi_0}}.
\end{split}
\end{align*}
  
  Our next result is taken from \cite[Lemma 2.12]{G&Zhao2022} and estimates moments of Dirichlet polynomials with cubic and quartic characters. 
\begin{lemma}
\label{lem:2.5}
With the notation as above, let $X$ and $y$ be real numbers and $m$ a positive integer.  For fixed $\varepsilon$ with $0<\varepsilon<1$ and any complex numbers $a(p)$, $j \in \{3, 4\}$, we have
 \begin{align*}
  \sumstar_{\substack{\chi, q \\ X/2< q \leq X}}  \left|\sum_{\substack{p \leq y}}\frac{a(p)\chi(p)}{p^{1/2}}\right|^{2m} 
  \ll_{\varepsilon}  X\sum^{\lceil m/j\rceil }_{i=0}m! \binom {m}{j i} \binom {j i}{i} & \binom {(j-1) i}{i} a_j
  \Big (\sum_{p \leq y} \frac {|a(p)|^2}{p}\Big )^{m-j i}\Big (\sum_{p \leq y} \frac {|a(p)|^j}{p^{\frac j2}}\Big )^{2i} \\
  & +X^{1/2+\varepsilon}y^{2m+2m\varepsilon} \Big( \sum_{p \leq y} \frac {|a(p)|^2}{p} \Big)^{m},
 \end{align*}
   where $\lceil x \rceil = \min \{ n \in \intz : n \geq x\}$,
   \[ a_3= \binom {2i}{i}\frac {i!}{36^i} \quad \mbox{and}  \quad a_4= \binom {3i}{i}\frac {(2i)!}{576^i}. \]
\end{lemma}

\subsection{Upper bounds for cubic and quartic Dirichlet $L$-functions}
\label{sec2.5}

  Let $\Lambda(n)$ denote the von Mangoldt function. We note the following result for $\log |L(\sigma+it,\chi)|$ from \cite[Proposition 2.3]{Munsch17} (see also \cite[Lemma 2.6]{G&Zhao2024-3}).
\begin{lemma}
\label{up}
 Assume the truth of GRH. Let $\sigma \geq 1/2$ and $t \in \rear$ and define $\log^+ t=\max\{ 0,\log t\}$. For any primitive Dirichlet character $\chi$ modulo $q$ and any $x \geq 2$, we have
\begin{equation} \label{variant}
    \log |L(\sigma+it, \chi)| \leq \Re \sum_{n\leq x}\frac{\chi(n)\Lambda(n)}{n^{1/2+\max(\sigma-1/2, 1/\log x) +it} \log n}\frac{\log x/n}{\log x}+\frac{\log q+\log^+ t}{\log x}+O\Big( \frac{1}{\log x} \Big).
\end{equation}
\end{lemma}

The above lemma, together with arguments in the proof of \cite[Proposition 2]{Szab} leads to the following bound for a weighted average of $\log |L(\sigma+it,\chi)|$ over various $t$'s.
\begin{lemma}
\label{lemmasum}
 Assume the truth of GRH. Let $k$ be a positive integer and $Q$, $a_1,a_2,\ldots, a_{k}$ be fixed positive real constants, $x\geq 2$.  Set $a:=a_1+\cdots+ a_{k}$.  Suppose $X$ is a large real number,  $t_1,\ldots, t_{k}$ are fixed real numbers satisfying $|t_i|\leq X^Q$. For any integer $n$, set
$$h(n):=\frac{1}{2} \Big( \sum^{k}_{m=1}a_mn^{-it_m} \Big).$$
Then, for any primitive Dirichlet character $\chi$ modulo $q \leq X$ and $0 \leq \sigma -1/2 \ll 1/\log x$, 
\begin{align}
\label{mainupper}
\begin{split}
 \sum^{k}_{m=1} a_m\log |L(\sigma+it_m,\chi)| \leq  2\Re \sum_{p\leq x} & \frac{h(p)\chi(p)}{p^{1/2+\max(\sigma-1/2, 1/\log x)}}\frac{\log x/p}{\log x} \\
 & +\Re\sum_{p\leq \min (x^{1/2}, \log X)} \frac{h(p^2)\chi(p^2)}{p}+(Q+1)a\frac{\log X}{\log x}+O(1).
\end{split}
\end{align}
\end{lemma}

\begin{proof}
The first term on the right-hand side of \eqref{mainupper} comes from the primes in the sum on the right-hand side of \eqref{variant} and the prime cubes and higher powers therein contribute $O(1)$.  Arguing as in the proof of \cite[Lemma 2.10]{G&Zhao2022}, the prime squares give rise to the second term on the right-hand side of \eqref{mainupper}.
\end{proof}

The following lemma estimates from above moments of cubic and quartic $L$-functions and can be obtained by modifying the proof of \cite[Theorem 2]{Harper}.
\begin{lemma}
\label{prop: upperbound}
With the notation as above and the truth of GRH, let $k\geq 1$ be a fixed integer and ${\bf a}=(a_1,\ldots ,a_{k}),\ t=(t_1,\ldots ,t_{k})$
 be real $k$-tuples such that $a_i \geq 0$ and that $|t_i|\leq X^Q$ for all $i$. Then, for large real number $X$ and $\sigma \geq 1/2$,
\begin{align*}
   \sumstar\limits_{\chi, q} \big| L\big(\sigma+it_1,\chi \big) \big|^{a_1} \cdots \big| L\big(\sigma+it_{k},\chi \big) \big|^{a_{k}} \ll_{{\bf a}} &  X(\log X)^{O(1)}.
\end{align*}
\end{lemma}
\begin{proof}
  We consider only the cubic case here due to similarities of the proofs. The assertion of the lemma may be obtained by modifying the proofs of \cite[Proposition 5.2--5.3]{G&Zhao2022}.  However, as we do not strive for precise exponents in the powers of $\log X$ in the present lemma, we shall follow the proof of \cite[Proposition 1]{Szab}. For this, upon dividing into dyadic blocks and using H\"older’s inequality with the exponents $a_1/a, \cdots , a_k/a$, it suffices to show that for any fixed integer $w \geq1$ and any real  $|t|\leq X^Q$, 
\begin{align*}
   \sum_{\substack{(q,3)=1 \\ X/2 < q \leq X}}\;  \sumstar_{\substack{\chi \shortmod{q} \\ \chi^3 = \chi_0}} \big| L\big(\sigma+it,\chi \big) \big|^{2w}  \ll_{w} &  X(\log X)^{O(1)}.
\end{align*}   
   
From Lemmas~\ref{RS} and \ref{up}, for $x \leq X$, 
\begin{align}
\label{mainupper1}
\begin{split}
  \log |L(\sigma+it,\chi)| \leq  \Re \sum_{p\leq x} \frac{\chi(p)}{p^{1/2+\max(\sigma-1/2, 1/\log x)+it}}\frac{\log x/p}{\log x}
    +\frac 12\log \log X+(Q+1)\frac{\log X}{\log x}+O(1).
\end{split}
\end{align}

  Let $\mathcal{N}(V, X)$  be the number of primitive cubic Dirichlet characters $\chi$ modulo $q$ with $X/2 < q\leq X$ and $\log|L(\sigma+it, \chi)|\geq V$. 
  Suppose that $V \geq \max \{2\log \log X,10(Q+1)\}$ and $\chi$  is counted in $\mathcal{N}(V, X)$.  We set $x=X^{10(Q+1)/V}$ in \eqref{mainupper1}.  So for this character $\chi$,
\begin{align*}
\begin{split}
 \Re \sum_{p\leq x} \frac{\chi(p)}{p^{1/2+\max(\sigma-1/2, 1/\log x)+it}}\frac{\log x/p}{\log x} \geq \frac V2.
\end{split}
\end{align*}

We take $m = \lceil V/(100(Q+1))\rceil $ in Lemma \ref{lem:2.5}, getting
\begin{align*} 
\begin{split}
\Big( \frac {V}2 \Big)^{2m}  \mathcal{N}(V, X) 
 \leq & \sum_{\substack{(q,3)=1 \\ X/2<q \leq X}} \ \sumstar_{\substack{\chi \shortmod{q} \\ \chi^3 = \chi_0}} \Big| \Re \sum_{p\leq x} \frac{\chi(p)}{p^{1/2+\max(\sigma-1/2, 1/\log x)+it}}\frac{\log x/p}{\log x} \Big|^{2m} \\
 \ll & X\sum^{\lceil m/3\rceil }_{i=0}m! \binom {m}{3i} \binom {3i}{i}\binom {2i}{i}\frac {i!}{36^i} \Big (\sum_{p \leq x} \frac {1}{p}\Big )^{m-3i}\Big (\sum_{p \leq x} \frac {1}{p^{3/2}}\Big )^{2i}+X^{1/2+\varepsilon}x^{2m+2m\varepsilon} \Big( \sum_{p \leq x} \frac {1}{p} \Big)^{m}.
\end{split}
\end{align*}

  We argue as in the proof of \cite[Proposition 5.2]{G&Zhao2022} and apply Lemma \ref{RS}, obtaining, for some constant $C_0$, 
\begin{align} 
\label{Nvest}
\begin{split}
\Big( \frac {V}2 \Big)^{2m}  \mathcal{N}(V, X) 
 \ll & X\sum^{\lceil m/3\rceil }_{i=0}m! \binom {m}{3i} \binom {3i}{i}\binom {2i}{i}\frac {i!}{36^i} \Big (\sum_{p \leq x} \frac {1}{p}\Big )^{m-3i}\Big (\sum_{p \leq x} \frac {1}{p^{3/2}}\Big )^{2i}+X^{1/2+\varepsilon}x^{2m+2m\varepsilon} \Big( \sum_{p \leq x} \frac {1}{p} \Big)^{m} \\
 \ll & X m^2 \max_{0 \leq i \leq \lceil m/3\rceil }\Big (\frac {m}{e} \Big )^m \frac {(2m)^{3i}}{i^{2i}}\Big (\log \log X+C_0 \Big )^{m-3i} \\
 \ll & X m^2 \max_{0 \leq x \leq m/3+1 }\Big (\frac {m}{e} \Big )^m \frac {(2m)^{3x}}{x^{2x}}\Big (\log \log X+C_0 \Big )^{m-3x}.
\end{split}
\end{align}

   We set $y=2m/( (\log \log X+C_0 )$ to observe that the derivative of the function $f(x) = y^{3x}/x^{2x}$ only vanishes at $x_0=y^{3/2}/e$ and $f(x_0)=e^{2x_0}$. We may further assume that $x_0 \leq m/3+1$ for otherwise it does not contribute to the maximum value in the last expression of \eqref{Nvest}. It thus follows that $f(x_0) \leq e^{2m/3+2}$ in this case.  Hence the maximum value in the last expression of \eqref{Nvest} occurs either at $i=0$, $m/3+1$ or $x_0$, yielding 
\begin{align*}
 & \Big( \frac {V}2 \Big)^{2m}  \mathcal{N}(V, X) 
 \ll  X m^2\Big (\frac {m}{e} \Big )^m \max \Big ( \Big (\log \log X+C_0\Big )^m, e^{2m/3},  6^mm^{m/3} \Big ).
\end{align*}

From the above, we infer that if $V \geq e^{10000(Q+1)(k+1)}(\log \log X+C_0 )$, then
\begin{equation}
\label{equ:bd-S-2}
\mathcal{N}(V, X)  \ll X\operatorname{exp}\left(-4wV\right).
\end{equation}

   We now set $V_0=\max \{ e^{10000(Q+1)(w+1)}(\log \log X+C_0 ), 2\log \log X,10(Q+1)\}$, so that \eqref{equ:bd-S-2} holds for $V \geq V_0$. Note moreover $\mathcal{N}(V, X) \ll X$ by setting $c=1$ in Lemma~\ref{PropDirpoly}. Now partial summation renders that
\begin{align*}
   \sum_{\substack{(q,3)=1 \\ X/2 < q \leq X}}\;  \sumstar_{\substack{\chi \shortmod{q} \\ \chi^3 = \chi_0}} \big| L\big(\sigma+it,\chi \big) \big|^{2w} 
   \ll Xe^{2wV_0}+\sum^{\infty}_{V=V_0}\mathcal{N}(V, X)e^{(V+1)2w} \ll  X(\log X)^{O(1)}.
\end{align*}  
  This completes the proof of the lemma.
\end{proof}

\section{Proof of Theorem \ref{t1} }
\label{sec:upper bd}

  As the proofs are similar, we give only the proof for the cubic case here.  We may also assume that $X/2 < q \leq X$ upon dividing $q$ into dyadic blocks.   Setting $\sigma=1/2$ in \eqref{mainupper} and exponentiating both sides, the resulting expression renders
\begin{align}
\label{basicest}
\begin{split}
 \big| L\big( \tfrac12+it_1,\chi\big) \big|^{a_1} & \cdots \big| L\big( \tfrac12+it_{k},\chi \big) \big|^{a_{k}} \\
 \ll &  \exp \left(\left( \Re\sum_{\substack{  p \leq x }} \frac{2h(p)\chi (p)}{p^{1/2+1/\log x}}
 \frac{\log (x/p)}{\log x} +
 \Re \sum_{\substack{  p \leq  \min (x^{1/2}, \log X) }} \frac{h(p^2)\chi (p^2)}{p} 
 +(Q+1)a\frac{\log X}{\log x}\right) \right ).
\end{split}
 \end{align}
 
   We now follow the approach by A. J. Harper in \cite{Harper} to define for a large number $T$,
$$ \alpha_{0} = \frac{\log 2}{\log X}, \;\;\;\;\; \alpha_{i} = \frac{20^{i-1}}{(\log\log X)^{2}} \;\;\; \mbox{for all} \; i \geq 1, \quad
\mathcal{J} = \mathcal{J}_{k,X} = 1 + \max\{i : \alpha_{i} \leq 10^{-T} \} . $$

Moreover, set
\[ {\mathcal M}_{i,j}(\chi) = \sum_{X^{\alpha_{i-1}} < p \leq X^{\alpha_{i}}}  \frac{2h(p)\chi (p)}{p^{1/2+1/(\log X^{\alpha_{j}})}} \frac{\log (X^{\alpha_{j}}/p)}{\log X^{\alpha_{j}}}, \quad 1\leq i \leq j \leq \mathcal{J} , \]
and
\[ P_{m}(\chi )= \sum_{2^{m} < p \leq 2^{m+1}} \frac{h(p^2)\chi (p^2)}{p}, \quad  0 \leq m \leq \frac{\log\log X}{\log 2}. \]

Let $\mathcal{C}(X)$ be the set of primitive cubic Dirichlet characters of conductor $q$ with $X/2 < q \leq X$.  Also define, for $0 \leq j \leq \mathcal{J}$,
\begin{align*}
 \mathcal{S}(j) =& \left\{ \chi \in \mathcal{C}(X) : | {\mathcal M}_{i,l}(\chi)| \leq \alpha_{i}^{-3/4} \; \; \mbox{for all}  \; 1 \leq i \leq j, \; \mbox{and} \; i \leq l \leq \mathcal{J}, \right. \\
 & \hspace*{3cm} \left. \;\;\;\;\; \text{but }  | {\mathcal M}_{j+1,l}(\chi)| > \alpha_{j+1}^{-3/4} \; \text{ for some } j+1 \leq l \leq \mathcal{J} \right\} . \\
 \mathcal{S}(\mathcal{J}) =& \left\{ \chi \in \mathcal{C}(X) : |{\mathcal M}_{i, \mathcal{J}}(\chi)| \leq \alpha_{i}^{-3/4} \; \mbox{for all}  \; 1 \leq i \leq \mathcal{J} \right\}, \; \mbox{and} \\
\mathcal{P}(m) =&  \left\{ \chi \in \mathcal{C}(X) : | P_{m}(\chi)| > 2^{-m/10} , \; \text{but} \; | P_{n}(\chi)| \leq 2^{-n/10} \; \mbox{for all} \; m+1 \leq n \leq \frac{\log\log X}{\log 2} \right\}.
\end{align*}

    We shall set $x=X^{\alpha_j}$ for $j \geq 1$ in \eqref{basicest} in what follows, so that the second summation on the right-hand side of \eqref{basicest} is over $p \leq \log X$.  Then $| P_{n}(\chi)| \leq 2^{-n/10}$ for all $n$ if $\chi \not \in \mathcal{P}(m)$ for any $m$, which leads to
\[ \Re \sum_{\substack{  p \leq \log X }}  \frac{h(p^2)\chi (p^2)}{p} = O(1). \]
 As the treatment for case $\chi \not \in \mathcal{P}(m)$ for any $m$ is easier compared to the other cases, we may assume that $\chi \in \mathcal{P}(m)$ for some $m$.   We further note that
$$ \mathcal{P}(m)= \bigcup_{m=0}^{\log \log X/2}\bigcup_{j=0}^{ \mathcal{J}} \Big (\mathcal{S}(j)\bigcap \mathcal{P}(m) \Big ), $$
so that it suffices to show that
\begin{align}
\label{sumovermj}
\begin{split}
 \sum_{m=0}^{\log \log X/2}\sum_{j=0}^{\mathcal{J}} & \sum_{\chi \in \mathcal{S}(j)\bigcap \mathcal{P}(m)} \big| L\big( \tfrac12+it_1,\chi\big) \big|^{a_1} \cdots \big| L\big(\tfrac12+it_{k},\chi \big) \big|^{a_{k}} \\
 & \ll X(\log X)^{(a_1^2+\cdots +a_{k}^2)/4}\prod_{1\leq j<l\leq k}  \big|\zeta(1+i(t_j-t_l)+\tfrac 1{\log X}) \big|^{a_ja_l/2}.
\end{split}
\end{align}

   Observe that
\begin{align*}
\text{meas}(\mathcal{P}(m)) \leq \sum_{\substack{(q,3)=1}} \ \sumstar_{\substack{\chi \shortmod{q} \\ \chi^3 = \chi_0}}
\Big (2^{m/10} |P_m(\chi)| \Big )^{2\lceil 2^{m/2}\rceil }.
\end{align*}
We now apply Lemma \ref{lem:2.5} and recall the bound (see \cite[(5.3)]{G&Zhao2022})
\begin{align*}
  m! \binom {m}{3i} \binom {3i}{i}\binom {2i}{i}\frac {i!}{36^i} \leq  m^{4m/3+1}.
\end{align*}
  Together with the inequality $|h(p^2)| \leq a/2$, we infer that for $m$ large enough in terms of $a$,
\begin{align}
\label{Pmest}
\begin{split}
  \text{meas}(\mathcal{P}(m)) \ll & X(2^{m/10})^{2\lceil 2^{m/2}\rceil } \sum^{\lceil \lceil 2^{m/2}\rceil /3\rceil }_{i=0}
  \Big (\lceil 2^{m/2}\rceil \Big )^{4\lceil 2^{m/2}\rceil /3+1} \Big (\sum_{2^m < p } \frac {a^2}{4p^2}\Big )^{\lceil 2^{m/2}\rceil -3i}
  \Big (\sum_{2^m < p } \frac {a^3}{8p^3}\Big )^{2i} \\
   \ll & X 2^m (2^{2m/3+m/5})^{\lceil 2^{m/2}\rceil }\Big (\sum_{2^m < p } \frac {a^3}{p^2}\Big )^{\lceil 2^{m/2}\rceil } \ll X 2^m (a^32^{-m/3+m/5})^{2^{m/2}} \ll X  2^{-2^{m/2}}.
\end{split}
 \end{align}
Now the Cauchy-Schwarz inequality and Lemma \ref{prop: upperbound} yield that for $2^{m} \geq (\log\log X)^{3}$ and $X$ large enough,
\begin{align*}
 \sum_{\chi \in  \mathcal{P}(m)} & \big| L\big(\tfrac12+it_1,\chi\big) \big|^{a_1} \cdots \big| L\big(\tfrac12+it_{k},\chi \big) \big|^{a_{k}} \\
& \leq \left( \text{meas}(\mathcal{P}(m)) \cdot
\sum_{\substack{(q,3)=1 \\ X/2<q \leq X}}  \sumstar_{\substack{\chi \shortmod{q} \\ \chi^3 = \chi_0}}\big| L\big(\tfrac12+it_1,\chi\big) \big|^{2a_1} \cdots \big| L\big(\tfrac12+it_{k},\chi \big) \big|^{2a_{k}} \right)^{1/2}
 \\
& \ll \left( X \exp\left( -(\log 2)(\log\log X)^{3/2} \right) X (\log X)^{O(1)} \right)^{1/2} \ll X.
\end{align*}

The above implies  that we may also assume that $0 \leq m \leq (3/\log 2)\log\log\log X$.  Similarly,
\begin{align}
\label{S0est}
\begin{split}
\text{meas}(\mathcal{S}(0)) \ll & \sum_{\substack{(q,3)=1 }} \ \sumstar_{\substack{\chi \shortmod{q} \\ \chi^3 = \chi_0}}
\sum^{\mathcal{J}}_{l=1}
\Big ( \alpha^{3/4}_{1}{|\mathcal
M}_{1, l}(\chi)| \Big)^{2\lceil 1/(10\alpha_{1})\rceil }=
\sum^{\mathcal{J}}_{l=1}\sum_{\substack{(q,3)=1 }} \ \sumstar_{\substack{\chi \shortmod{q} \\ \chi^3 = \chi_0}}\Big ( \alpha^{3/4}_{1}{|\mathcal
M}_{1, l}(\chi)| \Big)^{2\lceil 1/(10\alpha_{1})\rceil } .
\end{split}
\end{align}
   Note that
\begin{align*}
 \mathcal{J} \leq \log\log\log X , \quad \alpha_{1} = \frac{1}{(\log\log X)^{2}}  \quad \mbox{and} \quad \sum_{p \leq X^{1/(\log\log X)^{2}}} \frac{1}{p} \leq \log\log X ,
\end{align*}
  where the last bound is a consequence of Lemma \ref{RS}.  We apply these estimates and Lemma \ref{lem:2.5} to evaluate the last sums in \eqref{S0est} above in a manner similar to the \eqref{Pmest}.  From this consideration, emerges the estimate
\begin{align*}
\text{meas}(\mathcal{S}(0)) \ll &
\mathcal{J}X e^{-1/\alpha_{1}}\ll X e^{-(\log\log X)^{2}/10}  .
\end{align*}
The Cauchy-Schwarz inequality and Lemma~\ref{prop: upperbound} lead to
\begin{align*}
 \sum_{\chi \in  \mathcal{S}(0)}  & \big| L\big(\tfrac12+it_1,\chi\big) \big|^{a_1} \cdots \big| L\big(\tfrac12+it_{k},\chi \big) \big|^{a_{k}}  \\
\leq &   \left( \text{meas}(\mathcal{S}(0)) \cdot
\sum_{\substack{(q,3)=1 \\ X/2<q \leq X}}  \sumstar_{\substack{\chi \shortmod{q} \\ \chi^3 = \chi_0}}\big| L\big(1/2+it_1,\chi\big) \big|^{2a_1} \cdots \big| L\big(1/2+it_{k},\chi \big) \big|^{2a_{k}} \right)^{1/2}
 \\
 \ll & \left( X \exp\left( -(\log\log X)^{2}/10 \right) X (\log X)^{O(1)} \right)^{1/2} \ll X.
\end{align*}

  Thus we may further assume that $j \geq 1$. Note that when $\chi \in \mathcal{S}(j)$, we set $x=X^{\alpha_j}$ in \eqref{basicest} to arrive at
\begin{align*}
\begin{split}
 & \big| L\big(\tfrac12+it_1,\chi\big) \big|^{a_1} \cdots \big| L\big(\tfrac12+it_{k},\chi \big) \big|^{a_{k}} \ll \exp \left(\frac {(Q+1)a}{\alpha_j} \right) \exp \Big (
 \Re\sum^j_{i=1}{\mathcal M}_{i,j}(\chi)+ \Re\sum^{\log \log X/2}_{l=0}P_l(\chi) \Big ).
\end{split}
 \end{align*}

   When restricting the sum of $\big| L\big(\tfrac12+it_1,\chi\big) \big|^{a_1} \cdots \big| L\big(\tfrac12+it_{k},\chi \big) \big|^{a_{k}}$ over $\mathcal{S}(j)\bigcap \mathcal{P}(m)$, our treatments below require us to separate the sums over $p \leq 2^{m+1}$ on the right-hand side of the above expression from those over $p>2^{m+1}$. For this, we note that if $\chi \in \mathcal{P}(m)$, then
\begin{align*}
\begin{split}
  \Re \sum_{  p \leq 2^{m+1}} & \frac{2h(p)\chi (p)}{p^{1/2+1/(\log X^{\alpha_{j}})}}  \frac{\log (X^{\alpha_{j}}/p)}{\log X^{\alpha_{j}}}+
  \Re \sum_{p \leq \log X} \frac{h(p^2)\chi (p^2)}{p} 
   \\
 \leq &\Re \sum_{  p \leq 2^{m+1}}  \frac{2h(p)\chi (p)}{p^{1/2+1/(\log X^{\alpha_{j}})}}  \frac{\log (X^{\alpha_{j}}/p)}{\log X^{\alpha_{j}}}+
 \Re \sum_{p \leq 2^{m+1}} \frac{h(p^2)\chi (p^2)}{p}  +O(1)  \leq  a2^{m/2+3}+O(1).
\end{split}
 \end{align*}

 It follows from the above that
\begin{align}
\label{LboundinSP}
\begin{split}
  \sum_{\chi \in \mathcal{S}(j)\bigcap \mathcal{P}(m)} & \big| L\big(\tfrac12+it_1,\chi\big) \big|^{a_1} \cdots \big| L\big(\tfrac12+it_{k},\chi \big) \big|^{a_{k}}   \\
   \ll & e^{a2^{m/2+4}} \sum_{\chi \in \mathcal{S}(j)\bigcap \mathcal{P}(m)}
 \exp \left(\frac {(Q+1)a}{\alpha_j} \right)\exp \Big (  \Re \sum_{ 2^{m+1}< p \leq X^{\alpha_j} }
 \frac{2h(p)\chi (p)}{p^{\tfrac{1}{2}+1/(\log X^{\alpha_{j}})}}  \frac{\log (X^{\alpha_{j}}/p)}{\log X^{\alpha_{j}}}\Big )
  \\
\ll &  e^{a2^{m/2+4}} \exp \left(\frac {(Q+1)a}{\alpha_j} \right)\sum_{\chi \in \mathcal{S}(j)} \Big (2^{m/10}|P_m(\chi)| \Big )^{2\lceil 2^{m/2}\rceil }
\exp \Big ( \Re{\mathcal M}'_{1,j}(\chi)+\Re \sum^j_{i=2}{\mathcal M}_{i,j}(\chi)\Big ),
\end{split}
 \end{align}
where
\begin{align*}
\begin{split}
 {\mathcal M}'_{1,j}(\chi)= \sum_{ 2^{m+1}< p \leq X^{\alpha_1} }
 \frac{2h(p)\chi (p)}{p^{1/2+1/\log X^{\alpha_{j}}}}  \frac{\log (X^{\alpha_{j}}/p)}{\log X^{\alpha_{j}}}.
\end{split}
 \end{align*}

    We note that if $0 \leq m \leq (3/\log 2)\log\log\log X$ and $X$ large enough, then
\begin{align*}
\begin{split}
 \Big| \sum_{ p< 2^{m+1}  }
 \frac{2h(p)\chi (p)}{p^{1/2+1/\log X^{\alpha_{j}}}}  \frac{\log (X^{\alpha_{j}}/p)}{\log X^{\alpha_{j}}} \Big| \leq a\sum_{ p< 2^{m+1}  }
 \frac{1}{\sqrt{p}} \leq \frac {100a \cdot 2^{m/2}}{m+1} \leq 100a(\log \log X)^{3/2}(\log \log \log X)^{-1},
\end{split}
 \end{align*}
where the last inequality follows from partial summation and \eqref{mertenpartialsummation}. \newline

  It follows from this that if $\chi \in \mathcal{S}(j)$ and $X$ large enough,
\begin{align*}
\begin{split}
| {\mathcal M}'_{1,j}(\chi) | \leq 100a(\log \log X)^{3/2}(\log \log \log X)^{-1}+|{\mathcal M}_{1,j}(\chi)| \leq 1.01\alpha^{-3/4}_1=1.01(\log \log X)^{3/2}.
\end{split}
 \end{align*}
 The above bound allows us to argue as in the proof of \cite[Lemma 5.2]{Kirila} (which essentially shows that the exponential function $e^x$ can be approximated by its truncated Tayler expansion depending on the size of $x$) to see that 
\begin{align*}
\begin{split}
\exp \Big ( \Re {\mathcal M}'_{1,j}(\chi)\Big )=& \exp \Big ( \tfrac 12 {\mathcal M}'_{1,j}(\chi)+\tfrac 12 \overline{{\mathcal M}'_{1,j}(\chi)} \Big ) \\
\ll &   E_{e^2a\alpha^{-3/4}_1}(\tfrac 12{\mathcal M}'_{1,j}(\chi))E_{e^2a\alpha^{-3/4}_1}(\tfrac 12\overline{{\mathcal M}'_{1,j}(\chi)}) = \Big | E_{e^2a\alpha^{-3/4}_1}(\tfrac 12{\mathcal M}'_{1,j}(\chi)) \Big |^2, 
\end{split}
 \end{align*}
    where for any real numbers $x$ and $\ell \geq 0$,
\begin{align*}
  E_{\ell}(x) = \sum_{j=0}^{\lceil\ell \rceil} \frac {x^{j}}{j!}.
\end{align*}

   We then estimate $\exp \Big ( \Re {\mathcal M}'_{1,j}(\chi)+ \Re\sum^j_{i=2}{\mathcal M}_{i,j}(\chi)\Big )$ similarly by noting that $|{\mathcal M}_{i, j}(\chi)|\leq  \alpha^{-3/4}_i$ for $\chi \in \mathcal{S}(j)$. This leads to 
\begin{align*}
\begin{split}
\exp \Big ( \Re {\mathcal M}'_{1,j}(\chi)+ \Re\sum^j_{i=2}{\mathcal M}_{i,j}(\chi)\Big ) \ll \Big | E_{e^2a\alpha^{-3/4}_1}(\tfrac 12{\mathcal M}'_{1,j}(\chi)) \Big |^2
\prod^j_{i=2}\Big | E_{e^2a\alpha^{-3/4}_i}(\tfrac 12{\mathcal M}_{i,j}(\chi)) \Big |^2.
\end{split}
 \end{align*}

  Let $W(t)$ be a non-negative smooth function supported on $(1/4, 3/2)$ such that $W(t) \gg 1$ for $t \in (1/2, 1)$. It thus follows from the description on $\mathcal{S}(j)$ that if $j \geq 1$,
\begin{align*}
\begin{split}
 \sum_{\chi \in \mathcal{S}(j)\bigcap \mathcal{P}(m)} & \big| L\big(\tfrac12+it_1,\chi\big) \big|^{a_1} \cdots \big| L\big(\tfrac12+it_{k},\chi \big) \big|^{a_{k}}  \\
 \ll & e^{a2^{m/2+4}} \exp \left(\frac {(Q+1)a}{\alpha_j} \right)
 \sum^{R}_{l=j+1} \sum_{\substack{(q,3)=1 \\ X/2<q \leq X}} \ \sumstar_{\substack{\chi \shortmod{q} \\ \chi^3 = \chi_0}}\Big (2^{m/10}|P_m(\chi)| \Big )
 ^{2\lceil 2^{m/2}\rceil } \\
& \hspace*{1cm} \times \exp \Big ( \Re {\mathcal M}'_{1,j}(\chi)+ \Re \sum^j_{i=2}{\mathcal M}_{i,j}(\chi)\Big )\Big ( \alpha^{3/4}_{j+1}\big|{\mathcal
M}_{j+1, l}(\chi)\big |\Big)^{2\lceil 1/(10\alpha_{j+1})\rceil } \\
\ll &  e^{a2^{m/2+4}}\exp \left(\frac {(Q+1)a}{\alpha_j} \right)\sum^{R}_{l=j+1}
\sum_{\substack{(q,3)=1 }} \ \sumstar_{\substack{\chi \shortmod{q} \\ \chi^3 = \chi_0}}
\Big (2^{m/10}|P_m(\chi)| \Big )^{2\lceil 2^{m/2}\rceil } \\
& \hspace*{1cm} \times \Big | E_{e^2a\alpha^{-3/4}_1}(\tfrac 12{\mathcal M}'_{1,j}(\chi)) \Big |^2
\prod^j_{i=2}\Big | E_{e^2a\alpha^{-3/4}_i}(\tfrac 12{\mathcal M}_{i,j}(\chi)) \Big |^2\Big ( \alpha^{3/4}_{j+1}\big |{\mathcal
M}_{j+1, l}(\chi)\big |\Big)^{2\lceil 1/(10\alpha_{j+1})\rceil }W \big( \frac q{X} \big) .
\end{split}
 \end{align*}

 Note that, for $1 \leq j \leq \mathcal{I}-1$,
\begin{align*}
\mathcal{I}-j \leq \frac{\log(1/\alpha_{j})}{\log 20} \quad \mbox{and} \quad \sum_{X^{\alpha_{j}} < p \leq X^{\alpha_{j+1}}} \frac{1}{p}
 = \log \alpha_{j+1} - \log \alpha_{j} + o(1) = \log 20 + o(1) \leq 10 .
\end{align*}
Therefore we argue in a manner similar to the proof of \cite[Theorem 3]{G&Zhao2022} upon using Lemma \ref{PropDirpoly}  and \eqref{Pmest}. Observe that $g(p)=1+O(1/p)$ for the function $g$ defined in \eqref{gc}. It follows that by taking $T$ large enough,
\begin{align*}
\begin{split}
  \sum^{ \mathcal{I}}_{l=j+1} \sum_{\substack{(q,3)=1 }} & \ \sumstar_{\substack{\chi \shortmod{q} \\ \chi^3 = \chi_0}}
\Big (2^{m/10}|P_m(\chi)| \Big )^{2\lceil 2^{m/2}\rceil } \\
& \hspace*{1cm} \times \Big | E_{e^2a\alpha^{-3/4}_1}(\tfrac 12{\mathcal M}'_{1,j}(\chi)) \Big |^2
\prod^j_{i=2}\Big | E_{e^2a\alpha^{-3/4}_i}(\tfrac 12{\mathcal M}_{i,j}(\chi)) \Big |^2\Big ( \alpha^{3/4}_{j+1}\big |{\mathcal
M}_{j+1, l}(\chi)\big |\Big)^{2\lceil 1/(10\alpha_{j+1})\rceil }W \big(\frac q{X} \big)  \\
\ll & X(\mathcal{I}-j)e^{-44a(Q+1)/\alpha_{j+1}} 2^m (2^{-2m/15})^{\lceil 2^{m/2}\rceil } \prod_{p \leq X^{\alpha_j}}\left( 1+\frac {|h(p)|^2}{p}+O \left( \frac 1{p^2} \right) \right) \\
\ll & X(\mathcal{I}-j)e^{-44a(Q+1)/\alpha_{j+1}} 2^m (2^{-2m/15})^{\lceil 2^{m/2}\rceil } \prod_{p \leq X}\left( 1+\frac {|h(p)|^2}{p}+O \left( \frac 1{p^2} \right) \right) \\
\ll & e^{-42a(Q+1)/\alpha_{j+1}} 2^m (2^{-2m/15})^{\lceil 2^{m/2}\rceil }\exp\left(\sum_{p \leq X} \frac {|h(p)|^2}{p}+O \left( \frac 1{p^2} \right) \right),
\end{split}
 \end{align*}
 where the well-known inequality $1+x \leq e^x$ for all $x \in \mr$ is used in the last step. \newline

 As $|h(p)|^2=\sum^k_{j=1}a^2_j/4+\sum_{1 \leq j<l\leq k}\frac {a_ja_l}{2}\cos (|t_j-t_l|\log p)$, we deduce from the above, \eqref{merten} and \eqref{mertenstype} that
\begin{align*}
\begin{split}
 \sum^{ \mathcal{I}}_{l=j+1}
\sum_{\substack{(q,3)=1}}  & \ \sumstar_{\substack{\chi \shortmod{q} \\ \chi^3 = \chi_0}}
\Big (2^{m/10}|P_m(\chi)| \Big )^{2\lceil 2^{m/2}\rceil } \\
& \hspace*{1cm} \times \Big | E_{e^2a\alpha^{-3/4}_1}(\tfrac 12{\mathcal M}'_{1,j}(\chi)) \Big |^2
\prod^j_{i=2}\Big | E_{e^2a\alpha^{-3/4}_i}(\tfrac 12{\mathcal M}_{i,j}(\chi)) \Big |^2\Big ( \alpha^{3/4}_{j+1}\big |{\mathcal
M}_{j+1, l}(\chi)\big |\Big)^{2\lceil 1/(10\alpha_{j+1})\rceil }W \big(\frac q{X}\big)  \\
\ll  &  e^{-42a(Q+1)/\alpha_{j+1}} 2^m (2^{-2m/15})^{\lceil 2^{m/2}\rceil }X(\log X)^{(a_1^2+\cdots +a_{k}^2)/4}\prod_{1\leq j<l\leq k}  \big|\zeta(1+i(t_j-t_l)+\tfrac 1{\log X}) \big|^{a_ja_l/2}.
\end{split}
 \end{align*}

   We then conclude from the above and \eqref{LboundinSP} that (by noting that $20/\alpha_{j+1}=1/\alpha_j$)
\begin{align*}
\begin{split}
 \sum_{\chi \in \mathcal{S}(j)\bigcap \mathcal{P}(m)} & \big| L\big(\tfrac12+it_1,\chi\big) \big|^{a_1} \cdots \big| L\big(\tfrac12+it_{k},\chi \big) \big|^{a_{k}} \\
\ll & e^{-a(Q+1)/(10\alpha_{j})}2^m  e^{a2^{m/2+4}} (2^{-2m/15})^{\lceil 2^{m/2}\rceil }X(\log X)^{(a_1^2+\cdots +a_{k}^2)/4}\prod_{1\leq j<l\leq k}  \big|\zeta(1+i(t_j-t_l)+\tfrac 1{\log X}) \big|^{a_ja_l/2}.
\end{split}
 \end{align*}

The sum of the right-hand side expression over $m$ and $j$ converges, leading to \eqref{sumovermj} and the proof of Theorem \ref{t1}.

\section{Sketch of the Proof of Theorem~\ref{quadraticmean}}

As Theorem \ref{quadraticmean} is proved in essentially the same way as the proof of the first part of \cite[Theorem 3]{Szab} given in Section 7.2 of \cite{Szab}, we shall only give a brief outline here. \newline

Let $\mathcal{C}$ be the set of characters being summed over in the outer sums of $S_{3,m}$ and $S_{4,m}$.  It suffices to prove
\begin{equation} \label{thm3bound}
 \sum_{\chi \in \mathcal{C}} \Big | \sum_{n}\chi(n)\Psi\big(\frac n{Y}\big)\Big |^{2m} \ll  XY^m(\log X)^{(m-1)^2} .
 \end{equation}
This can be proved in the same way, with some small modification, as \cite[Lemma 8]{Szab} which differs from the above only in that it is stated with a specific smooth $\Psi$ with compact support.  \newline

The crucial input in proving \eqref{thm3bound} is a bound for the following expression of the form
\[ \sum_{\chi \in \mathcal{C}} \Big( \int\limits_0^B \big| L (\tfrac12 + it , \chi ) \big| \dif t \Big)^{2k} . \]
Following the treatment in \cite[Proposition 3]{Szab}, in which the specific nature of the set over which $\chi$ is averaged plays no role, and using Corollary~\ref{cor1}, we get the desired result. \newline

\noindent{\bf Data availability statement.} Data sharing is not applicable to this article as no new data were created or analyzed in this study. \newline


\noindent{\bf Acknowledgments.}  P. G. is supported in part by NSFC grant 12471003 and L. Z. by the FRG grant PS71536 at the University of New South Wales.  The authors would like to thank the anonymous referee for his/her careful inspection of the paper and helpful suggestions.

\bibliography{biblio}
\bibliographystyle{amsxport}

\end{document}